\documentclass[10pt]{article}

\usepackage{amssymb}
\usepackage{amsfonts}
\usepackage{amsthm}
\usepackage{amsmath}
\usepackage{float}
\usepackage{bigints}
\usepackage{fullpage}

\usepackage{latexsym}
\usepackage{amssymb,amsmath,amsopn}
\usepackage[dvips]{graphicx}   
\usepackage{color,epsfig}      

\newtheorem{thm}{Theorem}
\newtheorem{lem}{Lemma}

\newtheorem{defn}{Definition}
\newtheorem{rem}{Remark}

\newtheorem{conj}{Conjecture}
\newtheorem{prob}{Problem}
\newtheorem{example}{Example}

\renewcommand{\Re}{\mathbb R}
\newcommand{\E}{\mathbb E}

\newcommand{\KK}{\mathbf K}
\newcommand{\PP}{\mathbf P}
\newcommand{\QQ}{\mathbf Q}
\newcommand{\TT}{\mathbf T}

\newcommand{\K}{\mathcal{K}}
\newcommand{\C}{\mathcal{C}}

\newcommand{\B}{\mathcal{B}}
\newcommand{\F}{\mathcal{F}}

\newcommand{\xx}{\mathbf x}
\newcommand{\yy}{\mathbf y}
\newcommand{\pp}{\mathbf p}
\newcommand{\qq}{\mathbf q}
\newcommand{\rr}{\mathbf r}
\newcommand{\uu}{\mathbf u}
\newcommand{\oo}{\mathbf o}
\newcommand{\Sph}{\mathbb{S}}

\DeclareMathOperator{\bd}{bd}

\DeclareMathOperator{\conv}{conv}

\DeclareMathOperator{\proj}{proj}

\begin{document}

\title{On non-separable families of positive homothetic convex bodies
}

\author{K\'{a}roly Bezdek\thanks{Partially supported by a Natural Sciences and
Engineering Research Council of Canada Discovery Grant.} and Zsolt L\'angi\thanks{Partially supported by the J\'anos Bolyai Research Scholarship of the Hungarian Academy of Sciences and the OTKA K\_16 grant 119495.}}

\maketitle

\begin{abstract}
A finite family $\B$ of balls with respect to an arbitrary norm in $\Re^d$ ($d\ge 2$) is called a non-separable family if there is no hyperplane disjoint from $\bigcup \B$ that strictly separates some elements of $\B$ from all the other elements of $\B$ in $\Re^d$. In this paper we prove that if $\B$ is a non-separable family of 
balls of radii $r_1, r_2,\ldots , r_n$ ($n\geq 2$) with respect to an arbitrary norm in $\Re^d$ ($d\ge 2$), then $\bigcup \B$ can be covered by a ball of radius $\sum_{i=1}^n r_i$. This was conjectured by Erd\H os for the Euclidean norm and was proved for that case by A. W. Goodman and R. E. Goodman [Amer. Math. Monthly 52 (1945), 494-498]. On the other hand,
in the same paper A. W. Goodman and R. E. Goodman conjectured that their theorem extends to arbitrary non-separable finite families of positive homothetic convex bodies in $\Re^d$, $d\ge2$. Besides giving a counterexample to their conjecture, we prove that conjecture under various additional conditions.
\end{abstract}

\section{Introduction}

In this paper we identify a $d$-dimensional affine space with $\Re ^d$. By $\|\cdot\|$ and $\langle \cdot , \cdot \rangle$ we denote the canonical Euclidean norm and the canonical inner product on $\Re ^d$. A $d$-dimensional convex body $\KK$ is a compact convex subset of $\Re^d$ with non-empty interior. Moreover, $\KK$ is called $\oo$-\textit{symmetric} if $\KK$ is symmetric about the origin $\oo$ in $\Re^d$. The Minkowski sum or simply the vector sum of two convex bodies $\KK,\KK'\subset \Re^d$ is defined by 
$\KK+\KK'=\{\mathbf{k}+\mathbf{k}'\ |\  \mathbf{k}\in \KK, \mathbf{k}'\in \KK'\}$.
A homothetic copy of $\KK$ is a set of the form $\KK'=\xx+\tau \KK$, where $\tau$ is a non-zero real number and $\xx\in \Re^d$. If $\tau>0$, then $\KK'$ is said to be a positive homothetic copy of $\KK$. Moreover, let $\conv (X)$ stand for the convex hull of $X \subseteq \Re^d$. 

In 1945 A. W. Goodman and R. E. Goodman (\cite{GG}) proved a conjecture of Erd\H os on {\it non-separable families} of circular disks in $\Re^2$. In order to state that result we need the following definition.

\begin{defn}\label{erdos}
Let $\KK$ be a convex body in $\Re^d$ and let $\K = \{ \xx_i + \tau_i \KK\ |\  \xx_i\in \Re^d, \tau_i>0, i=1,2,\ldots, n\}$, where $d\ge 2$ and $n\ge 2$.
Assume that $\K$ is a {\rm non-separable family} in short, an {\rm NS-family}, meaning that every hyperplane intersecting $\conv \left( \bigcup \K \right)$ intersects a member of $\K$ in $\Re^d$, i.e., there is no hyperplane disjoint from $\bigcup \K$ that strictly separates some elements of $\K$ from all the other elements of $\K$ in $\Re^d$. Then, let $\lambda(\K) > 0$ denote the smallest positive value $\lambda$ such that a translate of $\lambda \left( \sum\limits_{i=1}^n \tau_i \right) \KK$ covers $\bigcup \K$.
\end{defn}

If $F: \Re^d \to \Re^d$ is an arbitrary invertible affine map, then it is straightforward to see that $F( \K)=\{ F(\xx_i+ \tau_i \KK)\ |\  \xx_i\in \Re^d, \tau_i>0, i=1,2,\ldots, n\}$ is an NS-family of convex bodies homothetic to $F(\KK)$ with homothety ratios $\tau_1,\tau_2,\ldots ,\tau_n$ such that $\lambda(\K) = \lambda(F( \K))$. Based on this property we can restate the main result of \cite{GG} as follows.

\begin{thm}[Goodman-Goodman]\label{Goodman-Goodman-theorem}
If $\C$ is an arbitrary NS-family of finitely many homothetic ellipses in $\Re^2$, then $\lambda(\C)\le 1$.
\end{thm}

For a completely different proof of Theorem~\ref{Goodman-Goodman-theorem} we refer the interested reader to Theorem 6.1 and its proof in \cite{BeLi}. On the other hand, on page 498 of \cite{GG} A. W. Goodman and R. E. Goodman put forward the following conjecture.

\begin{conj}[Goodman-Goodman(1945)]\label{Goodman-Goodman-1945}
For every convex body $\KK$ in $\Re^d$ and every NS-family $\K = \{ \xx_i + \tau_i \KK\ |\  \xx_i\in \Re^d, \tau_i>0, i=1,2,\ldots, n\}$ 
the inequality $\lambda(\K) \leq 1$ holds for all $d\ge 2$ and $n\ge 2$.
\end{conj} 

We note that it is straightforward to prove Conjecture~\ref{Goodman-Goodman-1945} for $n={\rm card}(\K)=2$ and $d\ge 2$. However, in what follows we give a counterexample to Conjecture~\ref{Goodman-Goodman-1945} for all ${\rm card}(\K)\ge 3$ and $d\geq 2$. Furthermore, we give a sharp upper bound on $\lambda(\K)$ for all NS-families $\K$ with ${\rm card}(\K)=3$ in $\Re^2$. 
Next we show that $\lambda(\K) \leq d$ holds for all $d\ge 2$. Then generalizing Theorem~\ref{Goodman-Goodman-theorem} we prove Conjecture~\ref{Goodman-Goodman-1945} for all centrally symmetric convex bodies $\KK$ in $\Re^d$, $d\ge 2$.
Moreover, we prove Conjecture~\ref{Goodman-Goodman-1945} for natural subfamilies of the family of NS-families namely, for {\it $k$-impassable families} in short, {\it $k$-IP-families} of positive homothetic convex bodies in $\Re^d$ whenever $0\le k\le d-2$. 

We conclude this section by inviting the interested reader to further investigate the basic problem of Goodman-Goodman rephrased as follows.

\begin{prob}
Find $\sup_{\K} \lambda(\K)$ for any given $d\ge 2$, where $\K$ runs over the NS-families of finitely many positive homothetic copies of an arbitrary convex body $\KK$ in $\Re^d$. In
particular, is there an absolute constant $c>0$ such that $\sup_{\K} \lambda(\K)\leq c$ holds for all $d\ge 2$?
\end{prob}

The results of this paper imply that 
$\frac{2}{3}+\frac{2}{3\sqrt{3}} = 1.0515\ldots \leq\sup_{\K} \lambda(\K)\leq d\  {\rm for}\ {\rm all}\ d\ge2$.

\section{Counterexample to Conjecture~\ref{Goodman-Goodman-1945} for ${\rm card}(\K)\ge 3$ in $\Re^d$, $d\geq2$}

\begin{figure}[h]
\begin{center}
\includegraphics[width=0.3\textwidth]{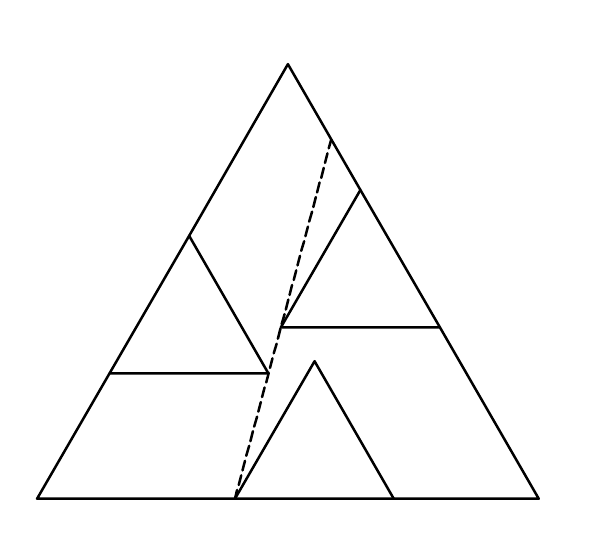}
\caption{A counterexample in the plane for three triangles}
\label{fig:counterexample2d_3p}
\end{center}
\end{figure}

\begin{example}\label{ex:counterexample2d_3p}
Place three regular triangles ${\mathcal T}=\{\TT_1, \TT_2, \TT_3\}$ of unit side lengths into a regular triangle $\TT$ of side length $2+\frac{2}{\sqrt{3}} = 3.154700\ldots > 3$ such that
\begin{itemize}
\item each side of $\TT$ contains a side of $\TT_i$, for $i=1,2,3$, respectively (cf. Figure~\ref{fig:counterexample2d_3p}),
\item for $i=1,2,3$, the vertices of $\TT_i$ contained in a side of $\TT$ divide this side into three segments of lengths $\frac{2}{3}+\frac{1}{\sqrt{3}}$, $1$, and $\frac{1}{3}+\frac{1}{\sqrt{3}}$, in counter-clockwise order.
\end{itemize}
A simple computation yields that the convex hull of any two of the small triangles touches the third triangle, and hence, no two of them can be strictly separated from the third one. Thus, $\lambda({\mathcal T})=\frac{2}{3} +\frac{2}{3\sqrt{3}} = 1.0515\ldots  >1$ and ${\mathcal T}$ is a counterexample to Conjecture~\ref{Goodman-Goodman-1945} for $n=3$ in $\Re^2$.
\end{example}

\begin{figure}[h]
\begin{center}
\includegraphics[width=0.3\textwidth]{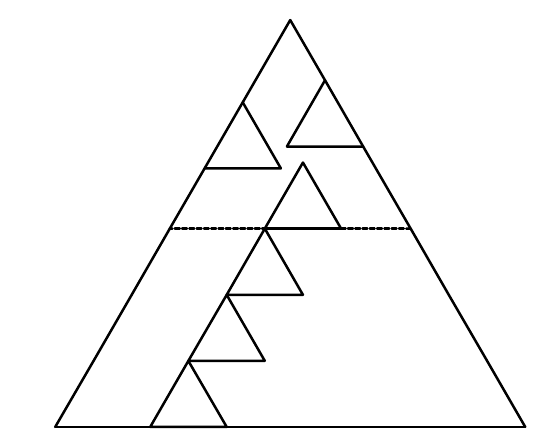}
\caption{A counterexample in the plane for $n$ triangles}
\label{fig:counterexample2d_np}
\end{center}
\end{figure}

\begin{rem}
Note that for any value $n \geq 4$, placing $n-3$ sufficiently small triangles inside $\TT$ in Example~\ref{ex:counterexample2d_3p} yields a counterexample to Conjecture~\ref{Goodman-Goodman-1945} for $n$ positive homothetic copies of a triangle in $\Re^2$. Furthermore, we can extend Example~\ref{ex:counterexample2d_3p}  using $n-3$ translates of the small triangles of Example~\ref{ex:counterexample2d_3p}, as in Figure~\ref{fig:counterexample2d_np}. As the small triangles have unit side length, the smallest triangle containing them has side length $n-1+\frac{2}{\sqrt{3}} > n$ and so, we get another counterexample to Conjecture~\ref{Goodman-Goodman-1945} using $n$ translates of a triangle in $\Re^2$.
\end{rem}

\begin{rem}\label{rem:counterexample_in_dd}
Let $d \geq 2$, and let $\lambda_s^d$ denote the supremum of $\lambda(\K)$, where $\K$ runs over the NS-families of finitely many positive homothetic $d$-simplices in $\Re^d$.
Then $\lambda_s^d$ is a non-decreasing sequence of $d$.
\end{rem}

\begin{proof}
Let $\K$ be an NS-family of finitely many $(d-1)$-simplices in a hyperplane of $\Re^d$. Clearly, we can extend the elements of $\K$ to homothetic $d$-simplices such that each element is a facet of its extension. Then, denoting this extended family by $\K'$, we have $\lambda(\K) = \lambda(\K')$, and $\K'$ is an NS-family.
\end{proof}

\begin{figure}[h]
\begin{center}
\includegraphics[width=0.3\textwidth]{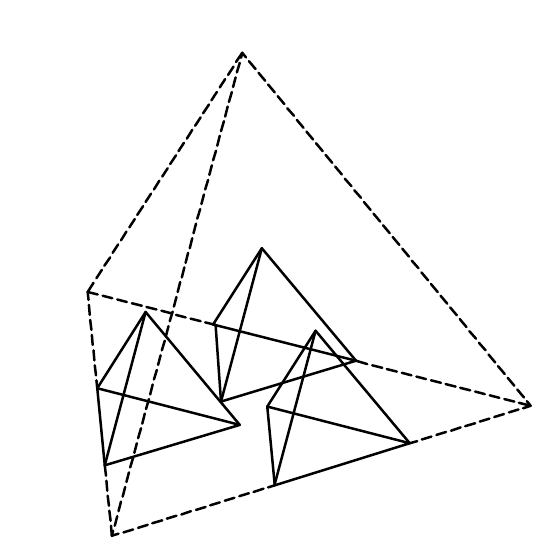}
\caption{A counterexample in $\Re^3$ for three tetrahedra}
\label{fig:counterexample3d_3p}
\end{center}
\end{figure}

Figure~\ref{fig:counterexample3d_3p} shows how to extend the configuration in Example~\ref{ex:counterexample2d_3p} to $\Re^3$, implying that $\lambda_s^d \geq \lambda_s^2\geq\frac{2}{3}+\frac{2}{3\sqrt{3}}= 1.0515\ldots>1$ for all $d\ge 3$.

\begin{rem}\label{reduction-to-simplices}
In fact, $\lambda_s^d=\sup_{\K} \lambda(\K)$ for all $d\ge 2$, where $\K$ runs over the NS-families of finitely many positive homothetic copies of an arbitrary convex body $\KK$ in $\Re^d$.
\end{rem}

\begin{proof}
Clearly, $\lambda_s^d\leq\sup_{\K} \lambda(\K)$. So, it is sufficient to show that for every convex body $\KK$ in $\Re^d$ and every NS-family $\K = \{ \xx_i + \tau_i \KK\ |\  \xx_i\in \Re^d, \tau_i>0, i=1,2,\ldots, n\}$ a translate of $\lambda_s^d \left( \sum\limits_{i=1}^n \tau_i \right) \KK$ covers $\bigcup \K$. Now, according to Lutwak's containment theorem (\cite{Lu}) if $\KK_1$ and $\KK_2$ are convex bodies in $\Re^d$ such that every circumscribed simplex of $\KK_2$ has a translate that covers $\KK_1$, then $\KK_2$ has a translate that covers $\KK_1$. (Here a circumscribed simplex of $\KK_2$ means a $d$-simplex of $\Re^d$ that contains $\KK_2$ such that each facet of the $d$-simplex meets $\KK_2$.) Thus, if $\Delta (\KK)$ is a circumscribed simplex of $\KK$, then $\lambda_s^d \left( \sum\limits_{i=1}^n \tau_i \right) \Delta(\KK)$ is a circumscribed simplex of $\lambda_s^d \left( \sum\limits_{i=1}^n \tau_i \right) \KK$ and $\xx_i + \tau_i \Delta(\KK)$ is a circumscribed simplex of $ \xx_i + \tau_i \KK$ for all $i=1,2,\ldots, n$. Furthermore, $\{ \xx_i + \tau_i \Delta(\KK)\ |\  \xx_i\in \Re^d, \tau_i>0, i=1,2,\ldots, n\}$ is an NS-family and therefore $\lambda_s^d \left( \sum\limits_{i=1}^n \tau_i \right) \Delta(\KK)$ has a translate that covers $\bigcup \{ \xx_i + \tau_i \Delta(\KK)\ |\  \xx_i\in \Re^d, \tau_i>0, i=1,2,\ldots, n\}\supseteq   \bigcup \K$, which completes our proof via Lutwak's containment theorem.
\end{proof}

\section{Upper bounding $\lambda(\K)$ for ${\rm card}(\K) =3$ in $\Re^2$}

\begin{thm}\label{thm:3homothetic}
Let $\KK$ be a convex body in $\Re^2$. If $\F$ is an NS-family of three positive homothetic copies of $\KK$, with homothety ratios $\tau_1, \tau_2, \tau_3$, respectively, then there is a translate of $\left( \frac{2}{3} +\frac{2}{3\sqrt{3}}\right)\left( \sum\limits_{i=1}^3 \tau_i \right)\KK$ containing $\F$. Furthermore, the constant $\frac{2}{3} +\frac{2}{3\sqrt{3}}= 1.0515\ldots$ is the smallest real number with this property.
\end{thm}

\begin{proof}
For simplicity, we set $\mu=\frac{2}{3} +\frac{2}{3\sqrt{3}}$.

First, note that by Example~\ref{ex:counterexample2d_3p}, there is a suitable $3$-element NS-family $\F$ satisfying $\lambda(\F) = \mu$. Thus, we need only to show that for any $3$-element NS-family, we have $\lambda(\F) \leq \mu$. On the other hand, using Lutwak's containment theorem in the same way as in Remark~\ref{reduction-to-simplices} we get that it is sufficient to prove Theorem~\ref{thm:3homothetic} when $\KK$ is a triangle $\TT$. Furthermore, as homothety ratios do not change under affine transformations, we may assume that $\TT$ is a regular triangle. So, let $\F = \{ \TT_1, \TT_2, \TT_3 \}$ such that for $i=1,2,3$, the homothety ratio of $\TT_i$ is $\tau_i$. Let $\overline{\TT}=\mu' (\sum\limits_{i=1}^3 \tau_i) \TT$ be the smallest positive homothetic copy of $\TT$ that contains $\F$. Without loss of generality, we may assume that $\mu' > 1$.

Since $\mu' >1$, it is easy to see that no two of $\TT_1$, $\TT_2$, and $\TT_3$ intersect.
For $i=1,2,3$, let $\pp_i$ and $\qq_i$ denote, respectively, the vertices of $\TT_i$ on $\bd \overline{\TT}$, in counter-clockwise order, and let $\rr_i$ denote the vertex of $\TT_i$ in the interior of $\overline{\TT}$ (cf. Figure~\ref{fig:triangle}). Without loss of generality, we may assume that the closed line segment $[\pp_3,\rr_1]$ intersects $\TT_2$. Since $\mu' > 1$, from this it follows that $[\pp_1,\rr_2]$ intersects $\TT_3$, and that $[\pp_2,\rr_3]$ intersects $\TT_1$. Clearly, if each intersection contains more than one point, then, moving the triangles a little apart we obtain an NS-family of three homothetic copies of $\TT$ that are contained only in a larger homothetic copy of $\TT$. We show that the same holds if at least one intersection contains more than one point. Assume, for example, that $\rr_3$ is contained in the interior of $\conv(\TT_1\cup \TT_2)$. Then, slightly moving $\TT_3$ or $\TT_1$ as well as $\TT_2$ such that $\overline{\TT}$ is fixed, we obtain a configuration where each intersection contains more than one point. Thus, in the following we assume that each of the triplets $\{\pp_1,\rr_3,\rr_2\}$, $\{\pp_2,\rr_1,\rr_3\}$, and $\{\pp_3, \rr_2,\rr_1 \}$ are collinear.

\begin{figure}[h]
\begin{center}
\includegraphics[width=0.28\textwidth]{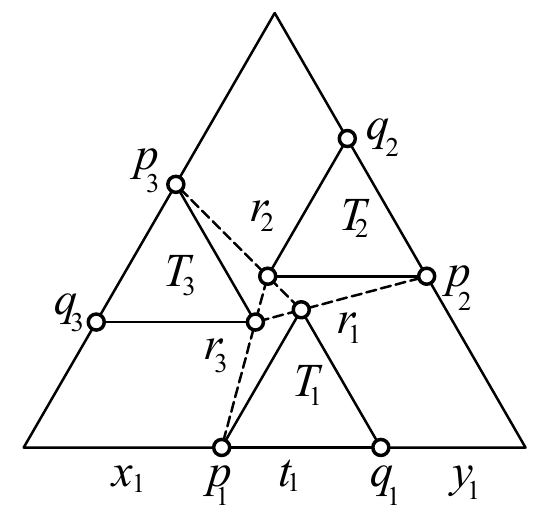}
\caption{An illustration for the proof of Theorem~\ref{thm:3homothetic}}
\label{fig:triangle}
\end{center}
\end{figure}

For $i=1,2,3$ let $\pp_i$ and $\qq_i$ divide the corresponding side of $\overline{\TT}$ into segments of length $x_i$, $t_i$, and $y_i$, in counter-clockwise order, and assume that $x_i+t_i+y_i=1$.
Then the collinearity of the three triplets yields that $t_3 x_2 + x_1y_3 - x_1x_2 -y_2y_3 = t_1 x_3 + x_2y_1 - x_2x_3-y_1y_3 = t_2 x_1 + x_3 y_2 - x_1x_3 - y_1y_2 = 0$.
Now, for the nine variables $x_i, t_i, y_i$, $i=1,2,3$ we have six conditions, and need to determine the minimum of the value of $f(t_1,t_2,t_3) = t_1+ t_2+t_3$ on the interval $(0,1)$.
Eliminating some of the variables, we may apply Lagrange's method with the remaining conditions to obtain the critical points of the expression $t_1+t_2+t_3$. We carried out this computation with a Maple 18.0 mathematical program, which provided five solutions.
(The command list we used can be found in the Appendix of this paper). A more careful analysis showed that two of the solutions were geometrically invalid (i.e., some of the parameters were outside the permitted range $(0,1)$). A third solution corresponded to the value $t_1+t_2+t_3 = 1$, which yields $\mu'=1$. The other two solutions were the following: $x_i = \frac{1}{4}+\frac{1}{4\sqrt{3}}$ and $t_i = \frac{3-\sqrt{3}}{4}$ for $i=1,2,3$, and $x_i = \frac{1}{4}-\frac{1}{4\sqrt{3}}$ and $t_i = \frac{3+\sqrt{3}}{4}$ for $i=1,2,3$.
The values of $\mu' = \frac{1}{t_1+t_2+t_3}$ at these points are $\frac{2}{3}+\frac{2}{3 \sqrt{3}}$ and $\frac{2}{3}-\frac{2}{3 \sqrt{3}}$, respectively, from which the assertion readily follows.
\end{proof}




Now, it is natural to ask the following. 

\begin{prob}\label{Goodman-Goodman-corrected}
Let $\KK$ be a convex body in $\Re^2$. If $\F$ is an NS-family of $k$ positive homothetic copies of $\KK$, with homothety ratios $\tau_1, \tau_2, \ldots , \tau_k$, respectively, and with $k\geq4$ , then prove or disprove that there is a translate of $\left( \frac{2}{3} +\frac{2}{3\sqrt{3}}\right)\left( \sum\limits_{i=1}^k \tau_i \right)\KK$ containing $\F$.
\end{prob}

We note that an NS-family of (large number of) convex bodies does not need to have NS-subfamilies as shown by the following example, which indicates the rather intricate nature of Problem~\ref{Goodman-Goodman-corrected}.

\begin{figure}[h]
\begin{center}
\includegraphics[width=0.25\textwidth]{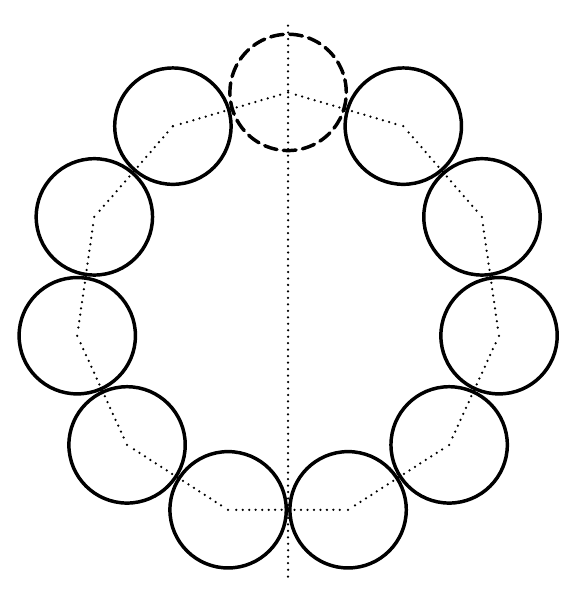}
\caption{An illustration for Example~\ref{ex:nohelly}: removing the dashed circle, the dotted line strictly separates the remaining circles}
\label{fig:nohelly}
\end{center}
\end{figure}  

\begin{example}\label{ex:nohelly}
Let $k\geq1$ and consider a regular $(2k+1)$-gon in $\Re^2$. Place congruent closed circular disks at the vertices of the polygon as centers, and choose their radius such that they are disjoint, but consecutive circular disks `almost' touch (cf. Figure~\ref{fig:nohelly}). Then they form an NS-family of circular disks. On the other hand, removing an arbitrary subfamily of at most $2k-1$ circular disks, the remaining (at least two) circular disks can be strictly separated by a line.
\end{example}

\section{Upper bounding $\lambda(\K)$ in $\Re^d$, $d\ge 2$}

\begin{thm}\label{thm:upperbound}
If $\K = \{ \xx_i + \tau_i \KK: \xx_i\in \Re^d, \tau_i>0, i=1,2,\ldots, n\}$ is an arbitrary NS-family of positive homothetic copies of the convex body $\KK$ in $\Re^d$, then $$\lambda(\K) \leq d$$
holds for all $n\ge 2$ and $d\ge 2$.
\end{thm}

For the proof of Theorem~\ref{thm:upperbound} we need the following simple observations.
For the first statement recall that, if ${\mathbf u} \in \Sph^{d-1}=\{\xx\in\Re^d\ |\ \|\xx\|=1\}$ and $\mathbf{Q}$ is a convex body in $\Re^d$, then the width of $\mathbf{Q}$ in the direction $\mathbf{u}$ is labelled by ${\rm width}_{\mathbf u}(\mathbf{Q})$ and it is equal to $\max \{\langle \xx,\mathbf{u}\rangle\ |\ \xx\in\mathbf{Q}\}-\min \{\langle \xx,\mathbf{u}\rangle\ |\ \xx\in\mathbf{Q}\}$. 

\begin{lem}\label{lem:width}
Let $\K = \{ \xx_i + \tau_i \KK: \xx_i\in \Re^d, \tau_i>0, i=1,2,\ldots, n\}$ be an arbitrary NS-family of positive homothetic copies of the convex body $\KK$ in $\Re^d, d\ge 2$. Then for any ${\mathbf u} \in \Sph^{d-1}$ we have ${\rm width}_{\mathbf u} \left( \conv \bigcup \K \right) \leq    \left(\sum\limits_{i=1}^n \tau_i\right)  {\rm width}_{\mathbf u}(\KK) $.
\end{lem}

\begin{proof}
As $\K$ is an NS-family therefore the orthogonal projection of $\conv \bigcup \K$ onto a line $L$ parallel to $\mathbf{u}$ is a line segment of length 
${\rm width}_{\mathbf u} \left( \conv \bigcup \K \right)$ which must be covered by the orthogonal projections of the convex bodies  $\xx_i + \tau_i \KK, i=1,2,\ldots, n$ onto $L$ having lengths equal to ${\rm width}_{\mathbf u}(\xx_i + \tau_i \KK)=\tau_i{\rm width}_{\mathbf u}(\KK), i=1,2,\ldots, n$. This yields the assertion.
\end{proof}

The following statement is well known. For the convenience of the reader we include its simple proof.

\begin{lem}\label{lem:approximation}
For any convex body $\mathbf Q$ in $\Re^d$, there exist $\xx, \yy \in \Re^d$ such that $\xx+\frac{2}{d+1} {\mathbf Q}_{0}\subseteq{\mathbf Q} \subseteq \yy+\frac{2d}{d+1} {\mathbf Q}_{0}$, where ${\mathbf Q}_0 = \frac{1}{2}({\mathbf Q}-{\mathbf Q})$ is the central symmetrization of ${\mathbf Q}$ with respect to the origin $\mathbf o$ in $\E^d$.
\end{lem}

\begin{proof}
Without loss of generality, let $\mathbf o$ be the center of a largest volume simplex inscribed in $\mathbf Q$. Then, clearly, ${\mathbf Q} \subseteq -d{\mathbf Q}$. From this it follows that $(d+1){\mathbf Q} \subseteq d({\mathbf Q}-{\mathbf Q})=2d{\mathbf Q}_0$. Thus, ${\mathbf Q} \subseteq \frac{2d}{d+1} {\mathbf Q}_0$. On the other hand, as $-{\mathbf Q} \subseteq d{\mathbf Q}$ therefore $2{\mathbf Q}_0={\mathbf Q}-{\mathbf Q} \subseteq (d+1){\mathbf Q}$. Hence, $\frac{2}{d+1} {\mathbf Q}_{0}\subseteq{\mathbf Q}$.
\end{proof}

\begin{proof}[Proof of Theorem~\ref{thm:upperbound}]
Let ${\mathbf M}=\conv \bigcup \K$. By Lemma~\ref{lem:width} we get that ${\rm width}_{\mathbf u} ( {\mathbf M}) \leq  \left(\sum\limits_{i=1}^n \tau_i\right)   {\rm width}_{\mathbf u}(\KK ) $ holds for every  ${\mathbf u} \in \Sph^{d-1}$. As  the central symmetrization does not change the width of a convex body in any direction therefore ${\mathbf M}_0\subseteq \left(\sum\limits_{i=1}^n \tau_i\right)\KK_{0}$. Finally, by Lemma~\ref{lem:approximation} there exist ${\mathbf m}, {\mathbf k} \in \Re^d$ such that 
$${\mathbf M} \subseteq {\mathbf m}+\frac{2d}{d+1} {\mathbf M}_{0} \subseteq {\mathbf m}+\frac{2d}{d+1} \left(\sum\limits_{i=1}^n \tau_i\right)\KK_{0} \subseteq {\mathbf m}+\frac{2d}{d+1} \left(\sum\limits_{i=1}^n \tau_i\right)\left(\frac{d+1}{2}({\mathbf k}+\KK) \right)$$
$$= \left({\mathbf m}+d\left(\sum\limits_{i=1}^n \tau_i\right){\mathbf k} \right) +d\left(\sum\limits_{i=1}^n \tau_i\right)\KK. $$\end{proof}

\begin{rem}
From the above proof of Theorem~\ref{thm:upperbound} it follows that if $\KK$ is an $\oo$-symmetric convex body in $\Re^d$, i.e., $\KK=\KK_{0}$, then $\lambda(\K) \leq \frac{2d}{d+1}$. In the next section, using a completely different approach, we give an optimal improvement on this inequality.
\end{rem}

\section{A proof of Conjecture~\ref{Goodman-Goodman-1945} for centrally symmetric convex bodies in $\Re^d$, $d\geq2$}

In this section we extend Theorem~\ref{Goodman-Goodman-theorem} to finite dimensional normed spaces by proving Conjecture~\ref{Goodman-Goodman-1945} for all centrally symmetric convex bodies $\KK$ in $\Re^d$, $d\ge 2$.

\begin{thm}\label{centrally symmetric convex bodies}
For every $\oo$-symmetric convex body $\KK_0$ and every NS-family $\K = \{ \xx_i + \tau_i \KK_0\ |\  \xx_i\in \Re^d, \tau_i>0, i=1,2,\ldots, n\}$ the inequality $\lambda(\K) \leq 1$ holds for all $d\ge 2$ and $n\ge 2$.
\end{thm}

Our proof is a close variant of the proof of Theorem 1 by Goodman-Goodman in \cite{GG} and it is
based on the following lemma, which is a strengthened version of Lemma 2 of \cite{GG}.

\begin{lem}\label{lem:GG}
Let $\F = \left\{ [x_i-\tau_i, x_i + \tau_i ] : \tau_i > 0, i=1,2,\ldots, n\right\}$ be a family of closed intervals in $\Re$ such that $\bigcup \F$ is  single closed interval in $\Re$. Let $x = \frac{\sum_{i=1}^n \tau_i x_i}{\sum_{i=1}^n \tau_i}$. Then the interval $\left[ x - \sum_{i=1}^n \tau_i , x + \sum_{i=1}^n \tau_i \right]$ covers $\bigcup \F$.
\end{lem}

\begin{proof}
We prove that the left end point of $\bigcup \F$ is covered by the interval $\left[ x - \sum_{i=1}^n \tau_i , x + \sum_{i=1}^n \tau_i \right]$; to prove that the right end point is covered, we may apply a similar argument.

Without loss of generality, we may assume that the left end point of the interval $\bigcup_{i=1}^n [x_i-\tau_i,x_i+\tau_i]$ is the origin $\oo$, moreover, that the numbering of the intervals $[x_i-\tau_i,x_i+\tau_i]$, $i=1,2,\ldots, n$ is in the order in which the left end points of these intervals occur on $\Re$ moving from left to right (in increasing order).
Based on this it is clear that
$$x_i\leq 2\tau_1+\ldots+2\tau_{i-1}+\tau_i,$$ where, for convenience we set $\tau_0 = 0$, and therefore
$$x=\frac{\sum_{i=1}^n x_i \tau_i}{\sum_{i=1}^n \tau_i}\leq \frac{\sum_{i=1}^n ( 2\tau_1+\ldots+2\tau_{i-1}+\tau_i ) \tau_i}{\sum_{i=1}^n \tau_i}=\frac{(\sum_{i=1}^n \tau_i  )^2}{\sum_{i=1}^n \tau_i}= \sum_{i=1}^n \tau_i, $$
showing that $\oo$ is indeed covered by the interval $\left[ x-\sum_{i=1}^n \tau_i, x+\sum_{i=1}^n \tau_i \right]$.
\end{proof}

\begin{proof}[Proof of Theorem~\ref{centrally symmetric convex bodies}]
Let $\xx = \frac{\sum_{i=1}^n \tau_i \xx_i}{\sum_{i=1}^n \tau_i}$, and set $\KK' = \xx + \left( \sum_{i=1}^n \tau_i \right) \KK_0$.
We prove that $\KK'$ covers $\bigcup \K$.

For any line $L$ through the origin $\oo$, let $\proj_L : \Re^d \to L$ denote the orthogonal projection onto $L$, and let $h_{\K} : \Sph^{d-1} \to \Re$ and $h_{\KK'} : \Sph^{d-1} \to \Re$ denote the support functions of $\conv\left( \bigcup \K \right)$ and $\KK'$, respectively.
Then $\proj_L \left( \bigcup \K \right)$ is a single interval, which, by Lemma~\ref{lem:GG}, is covered by $\proj_L \left( \KK' \right)$.
Thus, for any $\uu \in \Sph^{d-1}$, we have that $h_{\K}(\uu) \leq h_{\KK'}(\uu)$, which readily implies that $\bigcup \K \subseteq \KK'$. 
\end{proof}

\begin{rem}
If the positive homothetic convex bodies of $\K = \{ \xx_i + \tau_i \KK_0\ |\  \xx_i\in \Re^d, \tau_i>0, i=1,2,\ldots, n\}$ have pairwise disjoint interiors with their centers $\{\xx_i\ |\  i=1,2,\ldots, n\}$ lying on a line $L$ in $\Re^d$ such that the consecutive elements of $\K$ along $L$ touch each other, then $\K$ is an NS-family with $\lambda(\K) =1$.
\end{rem}

\section{A proof of Conjecture~\ref{Goodman-Goodman-1945} for $k$-impass\-ab\-le families in $\Re^d$ whenever $0\leq k\leq d-2$}

In this section we prove Conjecture~\ref{Goodman-Goodman-1945} for the following natural subfamilies of the family of NS-families, which following G. Fejes T\'oth and W. Kuperberg \cite{FeKu} we call {\it $k$-impassable families} in short, {\it $k$-IP-families} of (positive homothetic) convex bodies in $\Re^d$, where $0\le k\le d-2$.

\begin{defn}
Let $\K = \{ \xx_i + \tau_i \KK: \xx_i\in \Re^d, \tau_i>0, i=1,2,\ldots, n\}$ be a family of positive homothetic copies of the convex body $\KK$ in $\Re^d$ and let $0\le k\le d-1$. We say that $\K$ is a \emph{$k$-impassable arrangement}, in short, a \emph{$k$-IP-family} if every $k$-dimensional affine subspace of $\Re^d$ intersecting $\conv \left( \bigcup \K \right)$ intersects an element of $\K$. Let $\lambda_k(\K) > 0$ denote the smallest positive value $\lambda$ such that some translate of $\lambda \left( \sum\limits_{i=1}^n \tau_i \right) \KK$ covers $\bigcup \K$, where $\K$ is $k$-IP-family. A \emph{$(d-1)$-IP-family} is simply called an \emph{NS-family} and in that case $\lambda_{d-1}(\K)=\lambda(\K)$.   
\end{defn}

In order to state our theorem on $k$-IP-families of (positive homothetic) convex bodies in $\Re^d$ we need the following definitions and statement from Section 3.2 of \cite{Sch}.

\begin{defn}\label{defn:summand}
Let $\KK_1'$ and $\KK_2'$ be convex bodies in $\Re^d, d\ge 2$. We say that $\KK_2'$ is a \emph{summand} of $\KK_1'$ if there exists a convex body $\KK'$ in $\Re^d$ such that $\KK_2'+\KK'=\KK_1'$.
\end{defn}

\begin{defn}\label{defn:slidesfreely}
Let $\KK_1'$ and $\KK_2'$ be convex bodies in $\Re^d, d\ge 2$. We say that $\KK_2'$ \emph{slides freely} inside $\KK_1'$ if to each boundary point $\xx$ of $\KK_1'$ there exists a translation vector $\yy\in\Re^d$ such that $\xx\in\yy+\KK_2'\subseteq\KK_1'$.
\end{defn}

\begin{lem}\label{summand-1}
Let $\KK_1'$ and $\KK_2'$ be convex bodies in $\Re^d, d\ge 2$. Then $\KK_2'$ is a summand of $\KK_1'$ if and only if $\KK_2'$ slides freely inside $\KK_1'$.
\end{lem}

We prove Conjecture~\ref{Goodman-Goodman-1945} for $k$-IP-families of (positive homothetic) convex bodies in $\Re^d$ for all $0 \leq k \leq d-2$ in the following strong sense.

\begin{thm}\label{thm:smallk}
Let $\KK$ be a $d$-dimensional convex body and $\K=\{ \KK_i = \xx_i + \tau_i \KK: \xx_i\in \Re^d, \tau_i>0, i=1,2,\ldots, n\}$
be a $k$-IP family of positive homothetic copies of $\KK$ in $\Re^d$, where $0 \leq k \leq d-2$. Then $\conv\bigcup \K$ slides freely in $\left( \sum_{i=1}^n \tau_i \right) \KK$ (i.e., $\conv\bigcup \K$ is a summand of $\left( \sum_{i=1}^n \tau_i \right) \KK$) and therefore $\lambda_k(\K) \leq 1$.
\end{thm}

\begin{proof}
Set $\overline{\KK} = \left( \sum_{i=1}^n \tau_i \right) \KK$.
First we prove the assertion for $d=2$. Then $k=0$ and thus, $\bigcup \K$ is convex.

Assume that $\KK$ is a polygon. Then any side of $\bigcup \K$ has the same external normal vector as some side $S$ of $\KK$. Furthermore, the length of this side is at most $\left( \sum_{i=1}^n \tau_i \right) |S|$, where $|S|$ denotes the length of $S$. Hence, we may apply a theorem of Alexandrov \cite{Alexandrov}, that states that under this condition there is a translate of $\bigcup \K$ contained in $\overline{\KK}$. In addition, as is remarked in \cite{BB86}, a translation that translates a side of $\bigcup \K$ into the corresponding side of $\overline{\KK}$ translates also $\bigcup \K$ into $\overline{\KK}$. This yields that $\bigcup \K$ slides freely in $\overline{\KK}$, or, in particular, that $\lambda_0(\K) \leq 1$. 

Now we consider a general plane convex body $\KK$. First, we show that for any $i \neq j$, the intersection $(\KK_i \cap \KK_j) \cap (\bd \bigcup \K)$ consists of at most two segments. Indeed, let $\xx$ be a point of $\KK_i \cap \KK_j$ in $\bd \bigcup \K$, and let $L_{\xx}$ be a supporting line of $\bigcup \K$ at $\xx$. If $\KK_i$ and $\KK_j$ are translates, then the vector translating $\KK_i$ into $\KK_j$ is parallel to $L_{\xx}$, and, clearly, in this case every other point of $\KK_i \cap \KK_j$ and $\bd \bigcup \K$  is contained in either $L_{\xx}$, or the supporting line of $\bigcup \K$ parallel to $L_{\xx}$. If $\KK_i$ and $\KK_j$ are homothetic copies, then the center of the homothety transforming $\KK_i$ into $\KK_j$ lies on $L_{\xx}$. 
Thus, the claim follows from the fact that there is no point of $\Re^2$ lying on three distinct supporting lines of $\bigcup \K$.

Let $\xx$ be an arbitrary boundary point of $\overline{\KK}$. We show that there is vector $\yy \in \Re^2$ such that $\xx \in \yy+\bigcup \K \subseteq \overline{\KK}$. Without loss of generality, we may assume that $\xx$ is the origin of $\Re^2$, implying, in particular, that $\xx \in \bd \KK$.

Let $L_t$, $t=1,2,\ldots, m$ be supporting lines of $\bigcup \K$, in counterclockwise order in $\bd \bigcup \K$ that contain the points of $\bd \bigcup \K$ belonging to more than one member of $\K$. Let $\Gamma_t$ denote the arc of $\bd \bigcup \K$ between $L_t$ and $L_{t+1}$ for every value of $t$. Since every member of $\K$ intersects $\bd \bigcup \K$ in a closed set, there is a value $i_t$ of $i$ such that $\Gamma_t$ belongs to $\KK_{i_t}$.

For any positive integer $s$, let $\QQ^s$ be a convex polygon obtained as the intersection of some (closed) supporting halfplanes of $\bigcup \K$, including those bounded by the $L_t$'s. Similarly, let $\PP^s$ be the convex polygon, circumscribed about $\KK$, bounded by those translates of the supporting halfplanes defining $\QQ^s$ that support $\KK$. Since we can choose arbitrarily many supporting halfplanes of $\bigcup \K$, we may assume that $\PP^s \subseteq (1+1/s) \KK$ and that a sideline $L$ of $\PP^s$ contains $\xx$. Note that then $L$ is a supporting line of $\overline{\KK}$ as well. Finally, we set $\PP^s_i = \xx_i + \tau_i \PP^s$. Then $\QQ^s = \bigcup_{i=1}^n \PP^s_i$. Indeed, any point of $\QQ^s \setminus \bigcup \K$ is contained in a region bounded by two sidelines of $\QQ^s$ and an arc $\Gamma_t$ for some value of $t$, which is covered by $\PP^s_{i_t}$. Thus, we may apply our result for polygons, which yields that any translation transferring the side of $\QQ^s$ associated to $L$ into the side of $\left(\sum_{i=1}^n \tau_i \right) \PP^s$ contained in $L$, transfers also $\QQ^s$ into $\left(\sum_{i=1}^n \tau_i \right) \PP^s \subseteq (1+1/s) \overline{\KK}$. Let $\yy_s$ be an associated translation vector such that $\xx\in\yy_s + \bigcup \K \subseteq \yy_s + \QQ^s$. Choosing the limit $\yy$ of a convergent subsequence of these vectors we have that $\xx \in \yy + \bigcup \K \subseteq \overline{\KK}$, implying that $\bigcup \K$ slides freely in $\overline{\KK}$ and therefore $\lambda_0(\K) \leq 1$.

Now we prove Theorem~\ref{thm:smallk} for $d > 2$. For any $2$-dimensional linear subspace $H$ of $\Re^d$ let $\proj_H: \Re^d \to H$ denote the orthogonal projection onto $H$. As $\proj_H(\conv \bigcup\K)=\conv(\bigcup_{i=1}^{n}\proj_H(\KK_i))$ therefore $\{ \proj_H \KK_i : i=1,2,\ldots, n \}$ is a $0$-IP family in $H$. Thus, the above proof of Theorem~\ref{thm:smallk} for $d=2$ and Lemma~\ref{summand-1} imply that $\bigcup_{i=1}^{n}\proj_H(\KK_i)=\conv(\bigcup_{i=1}^{n}\proj_H(\KK_i))=\proj_H(\conv \bigcup\K)$ is a summand of $\proj_H \overline{\KK}$. Now we apply Lemma 3.2.6 from \cite{Sch}, which states that if $\KK_1'$ and $\KK_2'$ are two convex bodies in $\Re^d$, and $\proj_H \KK_1'$ is a summand of $\proj_H \KK_2'$ for all $2$-dimensional linear subspaces $H$ in $\Re^d$, then $\KK_1'$ is a summand of $\KK_2'$. Hence, $\conv\bigcup \K$ is a summand of $\overline{\KK}$ and therefore Lemma~\ref{summand-1} implies that $\conv\bigcup \K$ slides freely in $\overline{\KK}$, or in particular, that $\lambda_k(\K) \leq 1$.
\end{proof}

\begin{rem}
We note that for $0$-impassable families, Theorem~\ref{thm:smallk} can be proved in a simple way. Consider two intersecting homothetic copies
$\KK_1 = \xx_1 + \tau_1 \KK$ and $\KK_2 = \xx_2 + \tau_2 \KK$ of $\KK$. Let $\KK'$ be the smallest positive homothetic copy of $\KK$ that covers $\KK_1 \cup \KK_2$.Then there is a pair of parallel supporting hyperplanes $H_1$ and $H_2$ of $\KK'$ such that for $i=1,2$, $H_i \cap \KK_i \neq \emptyset$. Let $\uu \in \Sph^{d-1}$ be a normal vector of $H_1$ and $H_2$. Since $\mathrm{width}_{\uu} (\KK') \leq \mathrm{width}_{\uu}(\KK_1)+\mathrm{width}_{\uu}(\KK_2)$, it follows that the homothety ratio of $\KK'$ is at most $\tau_1+\tau_2$.
Thus, if $\K=\{ \KK_i = \xx_i + \tau_i \KK: \xx_i\in \Re^2, \tau_i>0, i=1,2,\ldots, n\}$ is $0$-impassable, applying this observation yields the assertion. In particular, this proves Theorem~\ref{thm:smallk} for $d=2$. On the other hand, to be able to apply the argument in the last paragraph of the proof of Theorem~\ref{thm:smallk}, we need to prove more for $d=2$, namely that $\bigcup \K$ is a summand of $\left( \sum_{i=1}^n \tau_i \right) \KK$.
\end{rem}

For strictly convex bodies one can do more.

\begin{thm}\label{thm:strictlyconvex}
Let $\KK$ be strictly convex body in $\Re^d, d\ge 2$ and $0 \leq k \leq d-2$. If $\K = \{ \KK_i=\xx_i + \tau_i \KK: \xx_i\in \Re^d, \tau_i>0, i=1,2,\ldots, n\}$ is $k$-IP-family of positive homothetic copies of $\KK$ in $\Re^d$, then $\bigcup \K = \xx_{i^*} + \tau_{i^*} \KK=\KK_{i^*}$ for $\tau_{i^*}=\max\{\tau_i\ |\ i=1,2,\dots, n\}$.
\end{thm}

\begin{proof}
Let $A$ be an arbitrary $k$-dimensional linear subspace $A$ of $\Re^d$ and let $A^\perp$ denote the orthogonal complement of $A$ in $\Re^d$. Moreover, let $\proj_{A^\perp}: \Re^d \to A^\perp$ denote the orthogonal projection onto $A^\perp$. Note that for every $A$, we have $\conv \left(\bigcup_{i=1}^n \proj_{A^\perp}(\KK_i) \right) = \bigcup_{i=1}^n  \proj_{A^\perp}(\KK_i) $. Furthermore, if $\KK$ is strictly convex, then so is $\proj_{A^\perp}(\KK)$.

We show that there is an $i^*=i^*(A)$ such that $\conv \left(\bigcup_{i=1}^n  \proj_{A^\perp}(\KK_i)  \right) = \proj_{A^\perp} (\KK_{i^*})$. Let $K^p_i = \proj_{A^\perp} (\KK_i)$, and assume, for contradiction, that $\bd (\bigcup_{i=1}^n K^p_i)$ is not covered by any of the $(K^p_i)$'s. Then there is a point $\pp \in \bd (\bigcup_{i=1}^n K^p_i)$, contained in $K^p_i \cap K^p_j$ for some $i \neq j$ such that $K^p_i \not\subseteq K^p_j \not\subseteq K^p_i$. Let $h$ be the (positive) homothety in $A^\perp$ that transforms $K^p_i$ into $K^p_j$, and let $\qq=h(\pp)$. Then, if $H$ is a closed supporting halfspace of $\bigcup_{i=1}^n K^p_i$ in $A^\perp$, then $H$ and $h(H)$ supports $K^p_j$ at $\pp$ and at $\qq$, respectively. As $H$ and $h(H)$ have the same outer unit normals, they coincide. Thus, if $\qq \neq \pp$, then the segment $[\pp,\qq]$ is contained in $K^p_j$, contradicting our assumption that $\KK$ is strictly convex. Thus, $\qq=\pp$ is the center of $h$, which implies that $K_i^p \subseteq K^p_j$, or that  $K_j^p \subseteq K^p_i$, a contradiction.

Hence, there exists an $i^*=i^*(A)$ such that $\conv \left(\bigcup_{i=1}^n \proj_{A^\perp}(\KK_i)  \right) =\proj_{A^\perp}  (\KK_{i^*})$. We note that $\conv \left(\bigcup_{i=1}^n  \proj_{\overline{A}^\perp}(\KK_i) \right)$ contains $ \proj_{\overline{A}^\perp}(\KK_{i^*})$ for all $k$-dimensional linear subspaces $\overline{A}$ of $\Re^d$, that is, a translate of $\lambda_{i^*}  \proj_{\overline{A}^\perp}(\KK)$, and thus, $i^*$ is independent of $A$. This yields that $\conv \bigcup \K = \KK_{i^*}$.
\end{proof}

\section{Acknowledgements}
The authors would like to thank the anonymous referees for careful reading and valuable comments
and proposing a shortcut in the proof of Theorem~\ref{centrally symmetric convex bodies}.

\section{Appendix}

For the convenience of the reader we include here the Maple code used in the the proof of Theorem~\ref{thm:3homothetic}.

{\small
\noindent
\texttt{>eq1 := t3*x2-x1*x2+x1*y3-y2*y3; eq2 := t1*x3-x2*x3+x2*y1-y1*y3;}

\noindent
\texttt{>eq3 := t2*x1-x1*x3+x3*y2-y1*y2;}

\noindent
\texttt{>Eq1 := simplify(subs(y1 = 1-x1-t1, subs(y2 = 1-x2-t2, subs(y3 = 1-x3-t3, eq1))));}

\noindent
\texttt{>Eq2 := simplify(subs(y1 = 1-x1-t1, subs(y2 = 1-x2-t2, subs(y3 = 1-x3-t3, eq2))));}

\noindent
\texttt{>Eq3 := simplify(subs(y1 = 1-x1-t1, subs(y2 = 1-x2-t2, subs(y3 = 1-x3-t3, eq3))));}

\noindent
\texttt{>T := solve([Eq1, Eq2], [t2, t3], explicit = true);}

\noindent
\texttt{>F := simplify(t1+subs(T[1][1], t2)+subs(T[1][2], t3)+a*subs(T[1][1], subs(T[1][2], Eq3)));}

\noindent
\texttt{>F1 := factor(simplify(diff(F, t1))); F2 := factor(simplify(diff(F, x1)));}

\noindent
\texttt{>F3 := factor(simplify(diff(F, x2))); F4 := factor(simplify(diff(F, x3)));}

\noindent
\texttt{>F5 := factor(simplify(diff(F, a)));}

\noindent
\texttt{>sols := solve([F1, F2, F3, F4, F5], [t1, x1, x2, x3, a], explicit = true, allsolutions = true);}
}

\noindent K\'aroly Bezdek \\
\small{Department of Mathematics and Statistics, University of Calgary, Calgary, Canada}\\
\small{Department of Mathematics, University of Pannonia, Veszpr\'em, Hungary}\\
\small{\texttt{bezdek@math.ucalgary.ca}}

\noindent and

\noindent Zsolt L\'angi \\
\small{Department of Geometry, Budapest University of Technology, Budapest, Hungary}\\
\small{\texttt{zlangi@math.bme.hu}}

\end{document}